\documentclass[11pt]{article}
\usepackage{amssymb}    
\usepackage{amsmath}    
\usepackage{amsthm} 

\usepackage{textcomp}
\usepackage{psfrag}
\usepackage{epsfig}
\usepackage{color}

\usepackage{latexsym}
\usepackage{amscd}

\newtheorem{prop}{Proposition}

\newtheorem{theorem}{Theorem}

\newtheorem{corollary}{Corollary}

\newtheorem{defin}{Definition}
\newtheorem{remark}{Remark}

\def\PROB {{\mathbb P}}
\def\EXP {{\mathbb E}}

\def\Z{{\mathbb Z}}

\def\R{{\mathbb R}}

\def\bX{{\bf X}}
\def\bA{{\bf A}}
\def\bB{{\bf B}}
\def\bC{{\bf C}}

\newcommand{\ZZ}{\mathbb{Z}}
\newcommand{\NN}{\mathbb{N}}

\begin{document}
\title{On the strong stability of ergodic iterations
\thanks{Attila Lovas and Mikl\'os R\'asonyi were supported by the National Research, Development and Innovation Office within the framework of the Thematic Excellence Program 2021; National Research subprogramme “Artificial intelligence, large networks, data security: mathematical foundation and applications” and also by the grant K 143529.}}
\author{
L\'aszl\'o Gy\"orfi \and Attila Lovas \and Mikl\'os R\'asonyi}

\maketitle{}

\begin{abstract}
We revisit processes generated by iterated random functions driven by a stationary and ergodic sequence.
Such a process is called strongly stable if a random initialization exists, for which the process is stationary and ergodic,
and for any other initialization the difference of the two processes converges to zero almost surely.
Under some mild conditions on the corresponding recursive map,
without any condition on the driving sequence we show the strong stability of iterations.
Several applications are surveyed such as generalized autoregression and queuing.
Furthermore, new results are deduced for Langevin-type iterations with dependent noise and for multitype branching processes.
\end{abstract}

\section{Introduction}

We are studying stochastic processes defined by iterating random functions.
For a measurable function $F: \R^d\times \R^k\to \R^d$,
consider the following iteration:
set $X_n=X_n(v)$ such that $X_{0}=v$, with a vector $v\in\R^d$ and let
\begin{align}
\label{itrec}
X_{n+1}=F(X_{n},Z_{n+1}),
\end{align}
where the driving sequence $\{Z_i\}_{1}^{\infty}$ is a stochastic process with values in $\mathbb{R}^{k}$.

In the standard setup, $\{Z_i\}_{1}^{\infty}$ is independent and identically distributed (i.i.d.) and
so $\{X_i(v)\}_{0}^{\infty}$ is a homogeneous Markov process,
see Diaconis and  Freedman \cite{DiFr99} and Iosifescu \cite{Ios09}.
Furthermore, (\ref{itrec}) is called a {\em forward iteration}. If in (\ref{itrec}), $Z_1,\dots ,Z_n$ is replaced by $Z_n,\dots ,Z_1$,
then the resulting iteration $\tilde X_n$ is called the {\em backward iteration},
see Propp and Wilson \cite{PrWi96}, \cite{PrWi98}.
Clearly, $X_n$ and $\tilde X_n$ have the same distribution for each $n$.

In the present paper the main role is played by another type of iteration, called {\it negative iteration},
defined as follows:
For a $k\le 0$, let the random double array $X^{(k)}_n=X^{(k)}_n(0)$, $k\le n$ be defined such that $X^{(k)}_{k}=0$ and
\begin{align}
\label{iLLrec}
X^{(k)}_{n+1}=F(X^{(k)}_n,Z_{n+1}),\quad n\ge k,
\end{align}
i.e., the iteration starts at negative time $k$ with initial vector $0$.
This iteration scheme is also often used and appears e.g.\ in
Borovkov \cite{Bor98},
Borovkov and Foss \cite{BorFoss93},
Debaly and Truquet \cite{DeTr21},
Diaconis and  Freedman \cite{DiFr99},
Elton \cite{Elt90},
Foss and Konstantopoulos \cite{FossKonst2003},
Gy\"orfi and  Walk \cite{GyWa96},
Gy\"orfi and  Morvai \cite{GyMo02} and
Iosifescu \cite{Ios09}.

Under mild conditions, Diaconis and  Freedman \cite{DiFr99} proved that the backward
iteration $\tilde X_n$ is almost surely (a.s.) convergent
to a random vector $V$ with a distribution $\nu$,
which implies that the forward iteration $X_n$ has the limit distribution $\nu$.
As in the standard setup of Markov chains, if $X_0$ has distribution  $\nu$ (and it is
independent from the driving sequence $\{Z_i\}_{1}^{\infty}$), then $X_n$ will be a stationary Markov process.

More general schemes have also been considered, where $\{Z_i\}_{1}^{\infty}$
is merely stationary and ergodic, see e.g.\ Debaly and Truquet \cite{DeTr21}, Elton \cite{Elt90} and Iosifescu \cite{Ios09}.
In the papers of Borovkov and Foss \cite{BorFoss93}, Foss and Konstantopoulos \cite{FossKonst2003}, and also
in the monograph Borovkov \cite{Bor98} such processes are treated under the name
``stochastically recursive sequences''. We remark for later use that, by the Doob-Kolmogorov theorem, a
stationary sequence $\{Z_i\}_{1}^{\infty}$ can always be completed to a
sequence $\{Z_i\}_{-\infty}^{\infty}$, defined on all the integer lattice $\mathbb{Z}=\{0,\pm 1,\pm 2,\dots\}$.
We assume henceforth that this completion has been carried out.

Brandt, Franken and Lisek \cite{BrFrLi90} called the stationary process  $\{X'_i\}_{-\infty}^{\infty}$
a \emph{weak solution} of the iteration, if
there exists a  $\{Z'_i\}_{-\infty}^{\infty}$ such that $(X'_i,Z'_i)$ satisfies the recursion (\ref{itrec}), and
$\{Z_i\}_{-\infty}^{\infty}$ and $\{Z'_i\}_{-\infty}^{\infty}$ having the same distribution.
$\{X_i\}_{-\infty}^{\infty}$ is called a \emph{strong solution} if it is stationary and $(X_i,Z_i)$ satisfies
the recursion (\ref{itrec}).

In this paper, we study the strong solutions.
Using the a.s. limit of the negative iteration we construct such solutions under mild conditions.
Actually, we investigate the following novel concept of \emph{strong stability}.

\begin{defin}
\label{def1}
The class of random processes $\{X_n(v), v\in\R^d\}$ is called strongly stable, if
\begin{itemize}
\item[(I)]
there exists a random vector $V^*$ such that $\{X_i(V^*)\}_{0}^{\infty}$ is
stationary and ergodic,
\item[(II)]
and for any random vector $V$, $X_n(V)-X_n(V^*)\to 0$ a.s.
\end{itemize}
\end{defin}

Note that in the definition above, the random initial vector $V$ may depend on
the entire trajectory of $\{Z_i\}_{1}^{\infty}$. As a result, the concept of
strong stability may seem overly demanding. Furthermore, for integer valued processes from (II) follows that there is a random index $\tau$ such that for all $n>\tau$, we have $X_n(V)=X_n(V^*)$. In other words, $\{X_i(V)\}_{0}^{\infty}$
is \emph{forward coupled} with $\{X_i(V^*)\}_{0}^{\infty}$. This stronger notion of stability was introduced by Lindvall \cite{Lin92}, and discussed also in Foss and Tweedie \cite{FoTw98}.
Traditional proofs establishing the existence of a unique limiting distribution for Markov chains on Polish spaces under Doeblin's minorization condition involve representing transitions through iterated i.i.d. random maps and a coupling argument \cite{BhaWay2002}. In this way, it can be shown that the iteration starting from any possibly random initial value is forward coupled with its stationary counterpart.

The aim of this paper is to show the strong stability of $\{X_n(v), v\in\R^d\}$ in great generality, for several relevant models and important applications. As in Debaly and Truquet \cite{DeTr21}, under some mild conditions on the function $F$, we show that the a.s. limiting process
\begin{align*}
X_n^*=\lim_{k\to -\infty}X^{(k)}_n
\end{align*}
exists and is stationary and ergodic.
Brandt, Franken and Lisek \cite{BrFrLi90} had similar results in the particular case
of monotonic $F$, see the definition (\ref{M2}) below.

Our main results are stated and proved in Section \ref{sec2}. Generalized autoregressions, queuing systems and
generalized Langevin dynamics are surveyed in Sections \ref{sec3}, \ref{sec4}, \ref{sec5}, respectively.
Finally, Section \ref{sec6} treats multi-type generalized Galton-Watson processes.

\section{Iterated ergodic function systems}\label{sec2}

Defining
\begin{align*}
F_n(x):=F(x,Z_{n}),\ n\in\mathbb{Z}
\end{align*}
we can write
\begin{align*}
X_n^{(k)}(v)=F_n \circ \dots \circ F_{k+1}(v),\quad k\le 0,\,\,n\geq k,
\end{align*}
where the empty composition is defined as the identity function.

In the sequel, $|\cdot|$ will refer to the standard Euclidean norm on $\mathbb{R}^{d}$.
For a function $g:\mathbb{R}^{d}\to\mathbb{R}^{d}$, set
\begin{align*}
\|g\|=\sup_{x\ne y}\frac{|g(x)-g(y)|}{|x-y|}.
\end{align*}

The following theorem is an extension of Theorem 5.1 in \cite{DiFr99}
and
Theorem 6.2 in \cite{Ios09}. It is contained in Theorem 3 of Elton \cite{Elt90}
(except for proving (II)). We provide a proof for completeness.

\begin{theorem}
\label{gen}
Assume  that  $\{Z_i\}_{-\infty}^{\infty}$  is a stationary and ergodic sequence.
Suppose that
\begin{itemize}
\item[(i)]
\[
\EXP\{ (\log \|F_1\|)^+\}<\infty,
\]
\item[(ii)]
and for some $n$,
\begin{align}
\label{Cont}
\EXP\{ \log \|F_n \circ \dots \circ F_1\|\}<0.
\end{align}
\end{itemize}
Then the class $\{X_n(v), v\in\R^d\}$ is strongly stable.
\end{theorem}
Notice that \eqref{Cont} is a sort of long run contraction condition here.

\begin{proof}[Proof of Theorem \ref{gen}.]
{For the stationary and ergodic process $\mathbf{Z}=\{Z_i\}_{-\infty}^{\infty}$, let $f_n$,
$n=1,2,\dots$ be vector valued functions
such that $f_i(T^i\mathbf{Z})=X_i(0)$, where $T$ stands for the shift transformation.
Let's calculate $f_n(\mathbf{Z})$.
If the process $\{X^{(k)}_n\}$ is defined in (\ref{iLLrec}), then
\begin{align*}
f_n(\mathbf{Z})=X^{(-n)}_{0}(0),
\end{align*}
i.e., $X^{(-n)}_{0}(0)$ is the value of the process at time $0$, when the process started
at negative time $-n$ with the $0$ vector.\\
We show that under the conditions (i) and (ii),
\begin{align}
\label{GAP1}
X^{(-n)}_{0}(0) \mbox{ is a.s. convergent to a random vector } V^*.
\end{align}
It will be clear that $V^{*}=f(\mathbf{Z})$ for some suitable functional $f$ so we
will in fact show
\begin{align}
\label{Br1}
f_n(\mathbf{Z})\to f(\mathbf{Z})\mbox{ a.s.}
\end{align}
As for (\ref{GAP1}), we show that this sequence is a.s. a Cauchy sequence, namely even
\begin{align*}
\sum_{n=1}^{\infty}|X^{(-n)}_{0}(0)-X^{(-n-1)}_{0}(0)|<\infty
\end{align*}
holds a.s. Notice that iterating (\ref{iLLrec}) yields
\begin{align*}
X^{(-n-1)}_{-n}(0)
&=F(X^{(-n-1)}_{-n-1}(0),Z_{-n})
=F(0,Z_{-n})
=X^{(-n)}_{-n}(F(0,Z_{-n})),
\end{align*}}
so
\begin{align*}
X^{(-n)}_{0}(0)-X^{(-n-1)}_{0}(0)
&=
X^{(-n)}_{0}(0)-X^{(-n)}_{0}(F(0,Z_{-n}))\\
&= F_{0} \circ \dots \circ F_{-n+1}(0-F(0,Z_{-n})).
\end{align*}
Thus,
\begin{align*}
|X^{(-n)}_{0}(0)-X^{(-n-1)}_{0}(0)|
&\le \|F_{0} \circ \dots \circ F_{-n+1} \|\cdot |F(0,Z_{-n})|.
\end{align*}
We will show that the ergodicity of $\{Z_i\}_{-\infty}^{\infty}$  together with (i) and (ii) implies
\begin{align}
\label{**}
\sum_{n=1}^{\infty}\|F_{0} \circ \dots \circ F_{-n+1} \|\cdot |F(0,Z_{-n})|<\infty
\end{align}
a.s. and so (\ref{GAP1}) is verified.\\
In the sequel, the key ingredient is Proposition 2 in Elton \cite{Elt90}, which is the extension of
F\"urstenberg and  Kesten  \cite{FuKe60}.
For proving (\ref{**}), note that by Proposition 2 in \cite{Elt90}, the condition (\ref{Cont}) implies that
the sequence
\begin{align*}
E_n
:= \frac 1n \EXP\{\log \Vert F_{0} \circ \dots \circ F_{-n+1} \Vert \}
\to E,\ n\to\infty
\end{align*}
with $E<0$ such that
\begin{align}
\label{EE}
\frac 1n \log \Vert F_{0} \circ \dots \circ F_{-n+1} \Vert
\to E
\end{align}
a.s.
Note that $E$ is called Lyapunov exponent.
Next, we argue as in Proposition 6.1 in Iosifescu \cite{Ios09}.
One has that
\begin{align*}
\Vert F_{0} \circ \dots \circ F_{-n+1} \Vert
&=
e^{n \frac 1n \log \Vert F_{0} \circ \dots \circ F_{-n+1} \Vert}.
\end{align*}
(\ref{EE}) implies that there are a random integer $n_0$ and $a>0$ such that, for all $n\ge n_0$,
\begin{align*}
\frac 1n \log \Vert F_{0} \circ \dots \circ F_{-n+1} \Vert
&\le
-a<0.
\end{align*}
Thus
\begin{align*}
\sum_{n=n_0}^{\infty}\|F_{0} \circ \dots \circ F_{-n+1} \|\cdot |F(0,Z_{-n})|
&\le
\sum_{n=n_0}^{\infty}|F(0,Z_{-n})|e^{-na}.
\end{align*}
Since $\lambda\ln^{+}|F(0,Z_{-1})|$ is integrable for all $\lambda>0$, it follows
that $$
\sum_{n=1}^{\infty}P(\lambda\ln |F(0,Z_{-1})|>n)<\infty.{}
$$ Applying this observation
for $\lambda:=2/\alpha$ the Borel-Cantelli lemma implies that $|F(0,Z_{-1})|<e^{na/2}$ holds except for finitely many $n$
almost surely, which
implies $\sum_{n=n_0}^{\infty}|F(0,Z_{-n})|e^{-na}<\infty$ almost surely.\\
Because of $X_i(V^*)=f(T^i\mathbf{Z})$, $\{X_i(V^*)\}_{0}^{\infty}$ is stationary and ergodic so (I) in the Definition \ref{def1} is proved.
Furthermore,
\begin{align*}
\|X_n(V)-X_n(V^*)\|
&\le \|F_{n} \circ \dots \circ F_{1} \|\cdot |V-V^*|
\to 0
\end{align*}
a.s., as before. Thus, (II) is verified, too.
\end{proof}

Theorem \ref{gen} applies, in particular, under the one-step contraction condition \eqref{con} below.
\begin{prop}
\label{itthm}
Assume  that  $\{Z_i\}_{-\infty}^{\infty}$  is a stationary and ergodic sequence
such that the distribution of $Z_{1}$ is denoted by $\mu$.
Suppose that
\[
\int |F(0,z)|\mu(dz)<\infty,
\]
and
\[
|F(x,z)-F(x',z)|\le K_z |x-x'|
\]
with
\begin{align}
\label{con}
\EXP\{ \log K_{Z_1}\}<0.
\end{align}
Then, the class $\{X_n(v), v\in\R^d\}$ is strongly stable.
\end{prop}
\begin{proof}
This proposition is an easy consequence of Theorem \ref{gen}, since
\begin{align*}
\EXP\{ \log \|F_n \circ \dots \circ F_1\|\}
&\le
\EXP\{ \log (\|F_n\| \cdots \| F_1\|)\}\\
&=
n\EXP\{ \log \| F_1\|\}\\
&\le
n\EXP\{ \log K_{Z_1}\}\\
&< 0.
\end{align*}
\end{proof}

\begin{defin}
We say that the class of strongly stable random processes $\{X_n(v), v\in\R^d\}$ satisfies the strong
law of large numbers (SLLN), if
$\EXP\{V^*\}$ is well-defined and finite, and for any $v\in\R^d$,
\[
\lim_n \frac 1n \sum^{n}_{i=1} X_{i}(v)=\EXP\{V^*\}
\]
a.s.
\end{defin}

\begin{remark} Under the conditions of Theorem \ref{gen},
if $V^*$ has a well-defined and finite expectation $\EXP\{V^*\}$ then we have SLLN:
\begin{align*}
\left|\frac 1n \sum^{n}_{i=1} X_{i}(v)-\EXP\{V^*\}\right|
&\le
\left|\frac 1n \sum^{n}_{i=1} X_{i}(V^*)-\EXP\{V^*\}\right|
+
\frac 1n \sum^{n}_{i=1} |X_{i}(v)- X_{i}(V^*)|.
\end{align*}
By Birkhoff's ergodic theorem, the first term on the right
hand side tends to $0$ a.s., while the a.s. convergence of the second term follows from (II).
\end{remark}

\begin{remark} Now we discuss some conditions guaranteeing $\EXP\{|V^{*}|\}<\infty$.
By Fatou's lemma and the triangle inequality, we can write
$$
\EXP \{|V^*|\}
\le
\liminf_{n\to\infty}\EXP \{|X_0^{(-n)}(0)|\}
{\le}
\sum_{k=0}^{\infty} \EXP \{| X_0^{(-k)}(0) - X_0^{(-k+1)}(0)|\}
$$
which we can estimate further and obtain
$$
\EXP \{|V^*|\} \le
\sum_{k=0}^{\infty} \EXP \{\lVert F_{0}\circ\ldots\circ F_{-k+1}\rVert\cdot | F(0,Z_{-k})|\}.
$$
For the sake of simplicity, assume for the moment that $z\mapsto |F(0,z)|$ is bounded by some
constant $C$.
By stationarity, it is enough to investigate $\EXP \{\lVert F_{k}\circ\ldots\circ F_{1}\rVert\}$.
Here, either we can prescribe a ``contractivity in the long run''-type condition like
\begin{equation}\label{eq:longrun}
\limsup_{k\to\infty}\EXP^{1/k} \{\lVert F_{k}\circ\ldots\circ F_{1}\rVert\}<1,
\end{equation}
and then the $n$-th root test gives the desired result (that is, $\EXP\{|V^{*}|\}<\infty$), or,
by H\"older's inequality, we can write
$$
\EXP \{\lVert F_{k}\circ\ldots\circ F_{1}\rVert\} \le \prod_{j=1}^{k}  \EXP^{1/k}\{\lVert F_{j}\rVert^k\} = \EXP \{\lVert F_{1}\rVert^k\},
$$
and thus we have
$$\EXP \{|V^*|\} \le \sup_{z}|F(0,z)| \EXP\left\{\sum_{k=1}^{\infty}\lVert F_{1}\rVert^k\right\}
\leq C
\EXP\left\{\frac{\lVert F_{1}\rVert}{1-\lVert F_{1}\rVert}\right\}.
$$
We should assume here that $\lVert F_{1}\rVert<1$ a.s.,
moreover $\EXP\left\{\frac{\lVert F_{1}\rVert}{1-\lVert F_{1}\rVert}\right\}<\infty$
hence this approach looks much more restrictive than requiring \eqref{eq:longrun}.\\
As pointed out in Truquet \cite{TruquetMCRE2021}, the above long-time contractivity
condition \eqref{eq:longrun} is stronger than \eqref{Cont} in Theorem \ref{gen} or, equivalently, \eqref{EE}.
On the other hand, in the i.i.d. case, \eqref{eq:longrun} reduces to $\EXP\{\lVert F_1\rVert\}<1$.
If $\EXP\{\lVert F_1\rVert\}<1$ fails in the i.i.d. case then $E[V^{*}]=\infty$ can easily happen,
see the example in Remark \ref{honig} below.
\end{remark}

\begin{remark}\label{honig}
Concerning the strong law of large numbers, we should have to verify
\begin{align}
\label{GAP2}
\EXP\{\sup_n|X^{(-n)}_{0}(0)|\}<\infty.
\end{align}
Note that $X^{(0)}_{0}(0)=0$, and thus
\begin{align*}
\sup_n |X^{(-n)}_{0}(0)|\le
\sum_{k=1}^{\infty}|X^{(-k+1)}_{0}(0)-X^{(-k)}_{0}(0)|.
\end{align*}
One possibility would be to check
\begin{align*}
\sum_{k=0}^{\infty}\EXP\left\{|X^{(-k+1)}_{0}(0)-X^{(-k)}_{0}(0)|\right\}
\le
\sum_{k=0}^{\infty}\EXP\left\{|F(0,Z_{-k})|\prod_{j=-k+1}^0K_{Z_j}\right\}
<\infty.
\end{align*}
Here is a counterexample, however, that this is not true in general.
Let $d=k=1$ and $\{Z_i\}_{i\in\Z}$ be i.i.d. such that
\begin{align*}
\PROB \{Z_0=e^{-2}\}=2/3\quad \mbox{and}\quad \PROB \{Z_0=e^2\}=1/3.
\end{align*}
Put  $X_{0}=v>0$ and
\begin{align*}
X_{n+1}=Z_{n+1}\cdot X_{n}.
\end{align*}
Clearly $K_{Z_0}=Z_0$ with $\EXP \{\log K_{Z_0}\}=-2/3<0$ and so the conditions of Proposition \ref{itthm} are satisfied.
Furthermore, $\EXP \{K_{Z_0}\}=\EXP\{Z_{1}\}=2/3 e^{-2}+1/3 e^2>1$.
By independence,
\begin{align*}
\EXP\{X_{n+1}\}
&=\EXP\{Z_{n+1}\}\cdot \EXP\{X_{n}\}\\
&=\EXP\{Z_{1}\}^{n+1}\cdot v\\
&=(2/3 e^{-2}+1/3 e^2)^{n+1}\cdot v\\
&\to \infty,
\end{align*}
as $n\to\infty$.
Therefore, $\EXP\{V^*\}=\infty$.
\end{remark}

{Next, the contraction condition of Theorem \ref{gen} is replaced by a monotonicity assumption.}
We denote by $\mathcal{F}_{t}$ the sigma-algebra generated by
$Z_{j}$, $-\infty<j\le t$. Furthermore, we use the notation $\mathbb{R}^d_+$ for the positive orthant endowed with the usual coordinate-wise partial ordering i.e. $x\le y$ for $x,y\in\mathbb{R}^d$ when each coordinate of $x$ is less than or equal to the corresponding coordinate of $y$. In what follows, $|x|_p=[|x_1|^p+\ldots+|x_d|^p]^{1/p}$ stands for the usual $l_p$-norm on $\mathbb{R}^d$.

\begin{prop}\label{Citthm}
Assume  that  $\{Z_i\}_{-\infty}^{\infty}$  is a stationary and ergodic sequence, and $F:\mathbb{R}^d_+\times\mathbb{R}^k\to\mathbb{R}^d_+$ is monotonic in its first argument:
\begin{align}
\label{M2}
F(x,z)\le F(x',z) \mbox{ if } x\le x',
\end{align}
For a fixed $1\le p <\infty$, suppose that there exist a constant $0\leq \rho<1$ and $K>0$ such that
\begin{align}
\label{Fo}
\EXP[|F(x,Z_{1})|_p\mid \mathcal{F}_{0}]\leq \rho |x|_p+K
\end{align}
almost surely for all $x\in\mathbb{R}^d$,  and
\begin{equation}\label{kezdoveges}
\EXP[|F(0,Z_0)|_p]<\infty.
\end{equation}
Then, the class $\{X_n(v), v\in\R^d\}$ is strongly stable.
Furthermore, the strong law of large numbers (SLLN) is satisfied.
\end{prop}
Notice that \eqref{Fo} also implies, for all $t$,
\begin{align}
\label{Fo1}
\EXP[|F(x,Z_{t+1})|_p\mid \mathcal{F}_{t}]\leq \rho |x|_p+K
\end{align}
almost surely for all $x\in\mathbb{R}_+^d$. Borovkov \cite{Bor98}, Foss and  Tweedie \cite{FoTw98} and  Propp and Wilson \cite{PrWi96} studied the monotonic iteration under the fairly restrictive condition that the range of the iteration is a bounded set.

\begin{proof}
This proposition extends the Foster-Lyapunov stability argument to a non-Markovian setup. We apply
the notations in the proof of Theorem~\ref{itthm} such that verify the condition (\ref{Br1}) or,
equivalently, check (\ref{GAP1}). Since $F_n(\cdot):=F(\cdot,Z_n):\mathbb{R}_+^d\to\mathbb{R}_+^d$, $n\in\mathbb{Z}$ are order-preserving maps, and $F(0,Z_{-n})\ge 0$,
\begin{align*}
X^{(-n-1)}_{0}(0)-X^{(-n)}_{0}(0)
&=
X^{(-n)}_{0}(F(0,Z_{-n}))-X^{(-n)}_{0}(0) \\
&= F_0\circ\ldots F_{-n+1}(F(0,Z_{-n}))-F_0\circ\ldots F_{-n+1}(0)
\end{align*}
is non-negative, therefore $\left(X^{(-n)}_{0}(0)\right)_{n\in\mathbb{N}}$ is monotonically increasing and so (\ref{GAP1}) is verified.\\
As for (I), we need that $V^*=\lim_{n\to\infty} X^{(-n)}_{0}(0)$ takes finite values a.s., which would follow from $\EXP\{\sup_n|X^{(-n)}_{0}(0)|_p\}<\infty$. Easily seen that $x\le y$ implies $|x|_p\le |y|_p$ for $x,y\in\mathbb{R}_+^d$, and therefore by Beppo-Levi theorem and the monotonicity of $X^{(-n)}_{0}(0)$, $n\in\mathbb{N}$,
\begin{align*}
\EXP\{\sup_n |X^{(-n)}_{0}(0)|_p\}
&=
\sup_n\EXP\{|X^{(-n)}_{0}(0)|_p\}.
\end{align*}
From \eqref{Fo1} we get that
$$
\EXP[|X_{i+1}^{(-n)}(0)|_p\mid\mathcal{F}_{i}]=\EXP[|F(X_i^{(-n)}(0),Z_{i+1})|_p\mid\mathcal{F}_{i}]\leq \rho
|X_i^{(-n)}(0)|_p+K,
$$
$\ i\geq -n.$
Iterating this leads to
$$
\EXP[|X^{(-n)}_{0}(0)|_p]\leq \rho^{n}\EXP[|X_{-n}^{(n)}(0)|_p]+\sum_{j=0}^{n-1} K\rho^{j}\leq \frac{K}{1-\rho}<\infty,
$$
hence $\sup_n\EXP\{|X^{(-n)}_{0}(0)|_p\}$ is finite which completes the proof.
\end{proof}

\begin{remark}\label{multitipe}
The above Proposition holds with an analogous proof if we replace \eqref{Fo1} by
$$
\EXP[|F(x,Z_{t+1})|_p\mid \mathcal{F}_{t}]\leq \rho(t) |x|_p+K,
$$
where $\rho(t)$, $t\in\mathbb{Z}$ is a stationary process adapted to $\mathcal{F}_{t}$ and
$$
\limsup_{n\to\infty}\EXP^{1/n}\{\rho(1)\cdots \rho(n)\}<1.
$$

\end{remark}

\section{Generalized autoregression}\label{sec3}

In this section, $\| A\|$ will denote the operator norm of a matrix $A$.
Several authors investigated the iteration of matrix recursion:
set $X_n=X_n(v)$ such that $X_{0}=v$, and
\begin{align}
\label{arrec}
X_{n+1}=A_{n+1}X_{n}+B_{n+1},\quad n\ge 0,
\end{align}
where $\{(A_n,B_n)\}$ are i.i.d., $A_n$ is a $d\times d$ matrix and $B_n$ is a $d$ dimensional vector.
In this section we study a more general case of \eqref{arrec} where
the sequence $\{(A_n,B_n)\}$ is stationary and ergodic.
The minimal sufficient condition for the existence of a stationary solution
for \eqref{arrec} was proved in Brandt \cite{brandt}.
In the i.i.d. case, Bougerol and Picard \cite{BoPi92} showed
that those conditions are indeed minimal.


For the stationary and ergodic case, we now reprove the sufficiency part of Theorem 2.5 of Bougerol and Picard \cite{BoPi92}.

\begin{prop}
Assume that $\{(A_n,B_n)\}$ is stationary and ergodic such that
$\EXP \{\log^{+}\Vert A_{0}\Vert\}<\infty$ and $\EXP\{ \log^{+}|B_{0}|\} <\infty $.
If
\begin{equation}\label{lya}
\EXP\{\log ||A_{0}\ldots A_{-n}||\}<0
\end{equation}
for some $n$
then the class $\{X_n(v), v\in\R^d\}$ is strongly stable.
\end{prop}

\begin{proof}
Since $\Vert A_{n}A_{n-1}\dots A_{1} \Vert$ corresponds to $\Vert F_{n} \circ \dots \circ F_{1} \Vert$, we
can verify
the conditions of Theorem \ref{itthm} easily: \eqref{lya} implies $(ii)$ and the integrability
conditions of our Proposition imply $(i)$.
\end{proof}

\begin{remark}
Note that the argument above showed up first in  F\"urstenberg and Kesten \cite{FuKe60}.
They proved that if $\EXP \{\log^{+}\Vert A_{0}\Vert\}<\infty$, then
\begin{align*}
\lim_n \frac 1n \log \Vert A_{n}A_{n-1}\dots A_{1} \Vert = E
\end{align*}
a.s., where $E$ stands for the Lyapunov exponent.
\end{remark}

In the rest of this section we recall results about a similar iteration which are based on negative iterations
(though they cannot be treated by our results in this paper). Let us consider
the stochastic gradient method for least-squares regression,
when there are given observation sequences of random, symmetric and positive semi-definite  $d\times d$ matrices $A_{n}$,
and random $d$-dimensional vectors $V_{n}$ such that
\[
A = \EXP(A_{n})
\]
and
\[
V = \EXP(V_{n})
\]
$(n=0,\pm 1,\pm 2,\ldots )$.
If $A^{-1}$ exists then the aim is to estimate
\[
\vartheta  = A^{-1}V.
\]
For this reason, a stochastic gradient algorithm
with constant gain  is introduced:
set $X_{0}=v$, and
\begin{align}
\label{rec}
X_{n+1}=X_{n}-\lambda (A_{n+1}X_{n}-V_{n+1}),\quad n\ge 0,
\end{align}
followed by an averaging:
\[
\bar X_{n}= \frac 1n \sum^{n}_{i=1} X_{i}.
\]

If $\lambda$ depends on $n$, then $\bar X_{n}$ is called averaged stochastic approximation introduced by
Polyak \cite{Pol90} and Ruppert \cite{Rup88}.
If  the sequence $\{(A_{n},\, V_{n})\}_{-\infty}^{\infty}$ is i.i.d., then they proved the optimal rate of convergence of $\bar X_{n}$ to $\vartheta$.

\textcolor{black}{Gy\"orfi and  Walk \cite{GyWa96} studied the  case of dependent $\{(A_{n},\, V_{n})\}_{-\infty}^{\infty}$:}
Assume  that the sequence $\{(A_{n},\, V_{n})\}_{-\infty}^{\infty}$  is stationary and ergodic
such that  $\EXP \Vert A_{n} \Vert<\infty$, $\EXP | V_{n}| <\infty $ and $A$ is positive definite.
Then there is $\lambda _{0}>0$ such
that for all $0<\lambda <\lambda _{0}$
there exists a stationary and ergodic process $\{X_n^*\}_{-\infty}^{\infty}$ satisfying the recursion (\ref{rec}) and
\begin{align*}
\lim_n (X_n-X_n^*)=0
\end{align*}
a.s. Moreover,
\begin{align*}
\lim_n \bar X_n=\vartheta+\delta_{\lambda}
\end{align*}
a.s. with an asymptotic bias vector $\delta_{\lambda}$.
In \cite{GyWa96} there is a 3-dependent example of $\{(A_{n},\, V_{n})\}_{-\infty}^{\infty}$, where $\delta_{\lambda}\ne 0$.
Furthermore, under suitable mixing condition on $\{(A_{n},\, V_{n})\}_{-\infty}^{\infty}$, $|\delta_{\lambda}|$ is of order $\sqrt{\lambda}$.

\section{Lindley process}\label{sec4}

We recall some results from queuing theory. They do not follow from arguments of the present paper
but they are also based on negative iterations hence they provide one more illustration
for the usefulness of this technique.

\textcolor{black}{For $d=k=1$, consider the following iteration:
set $X_n=X_n(v)$ such that $X_{0}=v\geq 0$,}
\begin{align*}
X_{n+1}=(X_{n}+Z_{n+1})^+.
\end{align*}

The next proposition is an extension of the concept of strong stability. Let $\{X'_i\}_{0}^{\infty}$ be a stationary and ergodic sequence. Recall that $\{X_i\}_{0}^{\infty}$ is \emph{forward coupled} with $\{X'_i\}_{0}^{\infty}$, if there is a random index $\tau$ such that for all $n>\tau$, we have $X'_n=X_n$.

\begin{prop}
\label{GyMo}
(Gy\"orfi and  Morvai \cite{GyMo02}.)
Assume, that  $\{Z_i\}_{-\infty}^{\infty}$ is a stationary and ergodic sequence with $\EXP\{Z_1\}<0$.
Put
\begin{align*}
V^*=\sup_{j\le 0}(Z_j+\dots +Z_0)^+
\end{align*}
and
\begin{align*}
X'_n=X_n(V^*).
\end{align*}
Then $\{X'_i\}_{0}^{\infty}$ is stationary and ergodic and
$\{X_i\}_{0}^{\infty}$ is forward coupled with $\{X'_i\}_{0}^{\infty}$.\hfill $\Box$
\end{prop}

Theorem 4.1 in Diaconis and  Freedman \cite{DiFr99} is about this iteration, when $\{Z_i\}_{-\infty}^{\infty}$  is i.i.d.
It is noted there that the condition $\EXP\{Z_1\}<0$ can be weakened to
\begin{align*}
\sum_{j=-\infty}^0\frac{\PROB\{Z_j+\dots ,Z_0>0\} }{j}<\infty.
\end{align*}
We guess that this observation is valid for the ergodic case, too.

As an application of Proposition \ref{GyMo} consider the extension of the G/G/1 queueing model.
Let $X_n$ be the waiting time of the $n$-th arrival, $S_n$ be the service time
of the $n$-th arrival, and $T_{n+1}$ be the inter arrival time between the $(n+1)$-th and $n$-th arrivals.
Then, we get the recursion
\begin{align*}
X_{n+1}=(X_{n}-T_{n+1}+S_n)^+.
\end{align*}

\textcolor{black}{Loynes \cite{Loy62}} and Gy\"orfi and  Morvai \cite{GyMo02} studied the generalized G/G/1, where either the arrival times, or the service times, or both are not memoryless, see also the books by
Asmussen \cite{Asm03},
Baccelli and  Br\'emaud \cite{BaBr03},
Borovkov \cite{Bor76}, \cite{Bor98},
Brandt, Franken and Lisek \cite{BrFrLi90},
Ganesh,  O'Connell and  Wischick \cite{GaOCWi04}.

Proposition \ref{GyMo} implies that if $\{ Z_{i}\}:=\{S_{i-1}-T_i\}_{-\infty}^{\infty}$ is a stationary
and ergodic sequence with $\EXP\{S_0\}<\EXP\{T_1\}$, then $\{X_i\}_{0}^{\infty}$ is forward coupled with a
stationary and ergodic $\{X'_i\}_{0}^{\infty}$.

\section{The Langevin iteration}\label{sec5}

\textcolor{black}{For a measurable function $H: \R^d\times \R^m\to \R^d$,
the Langevin iteration is defined as follows:}
set $X_n=X_n(v)$ such that $X_{0}=v$, and
\begin{equation}
\label{LLrec}
X_{n+1}=X_{n}-\lambda H(X_n,Y_n)+\sqrt{2\lambda}N_{n+1},\quad n\ge 0,
\end{equation}
where $\{Y_i\}_{-\infty}^{\infty}$ and $\{N_i\}_{-\infty}^{\infty}$ are random sequences.
In the literature of stochastic approximation, $ \lambda >0 $ is called step size, while in machine learning.
it is called learning rate.

The simplest case is where the sequences $N_n$ and $Y_n$ are independent, $Y_n$ is i.i.d.\ and $N_n$ is i.i.d.\
standard $d$-dimensional Gaussian.
This algorithm was introduced in Welling and Teh \cite{wt} and later analysed by
a large corpus of literature which we cannot review here. It is called ``stochastic gradient Langevin
dynamics'' and it can be used for sampling from high-dimensional, not necessarily log-concave distributions
and for finding the global minimum of high-dimensional functionals.
In this context,
$Y_n$ repesents a data sequence (obtained, for instance, by averaging a big dataset over randomly chosen
minibatches), $N_n$ is artificially added noise to guarantee that the process does not get stuck near
local minima.

The case where the data sequence $Y_n$ may be a dependent stationary process (but $N_n$ is still i.i.d.\ standard Gaussian)
has been treated less extensively:
see Dalalyan \cite{dalalyan} and Bakhagen et al. \cite{6} for the convex and Chau et al. \cite{5} for the non-convex settings.

Another stream of literature, starting from Hairer \cite{hairer}, concentrated on
stochastic differential equations driven by coloured Gaussian noises.
The discrete-time case of difference equations was treated in Varvenne \cite{varvenne}.
This setting corresponds to the case where in \eqref{LLrec}, $Y_n$
is constant and $N_n$ is a dependent Gaussian sequence.

We know of no studies so far that allowed \emph{both} $Y_n$ and $N_n$ to be only stationary.
We manage to establish strong stability in this case, under reasonable assumptions.

Defining
\[
F(x,z)=x-\lambda H(x,y)+\sqrt{2\lambda}u,\ z=(y,u)
\]
and $Z_{i+1}:=(Y_i,N_{i+1})$, Proposition \ref{itthm} implies the following:
\begin{corollary}
Assume  that the sequence $\{(Y_i,N_{i+1})\}_{-\infty}^{\infty}$ is stationary and ergodic,
and for a $\lambda _{0}>0$ and for all $0<\lambda <\lambda _{0}$
\[
\EXP|-\lambda H(0,Y_1)+\sqrt{2\lambda}N_{2}|<\infty,
\]
and
\[
|x-\lambda H(x,z)-(x'-\lambda H(x',z))|\le K_z|x-x'|
\]
with (\ref{con}).
Then,  for all $0<\lambda <\lambda _{0}$,
the class $\{X_n(v), v\in\R^d\}$ is strongly stable.
\end{corollary}

\begin{remark}
\label{abo}
Next, we prove the convergence of the iterative scheme \eqref{LLrec} assuming only that $H$ satisfies
\begin{equation}
\label{eq:sconv}
\left\langle
\partial_1 H (x,y)v,v
\right\rangle
\ge
m (y) | v |^2
\quad
\text{and}
\quad
|
\partial_1 H (x,y)
|
\le
M(y)
\end{equation}
with measurable $m,M:\R^{m}\to [0,\infty)$.
(This is a parametric form of the usual \emph{strong convexity condition}. One can replace it by a so-called \emph{dissipativity} condition and hence extend the analysis beyond the convex case. However, this direction requires a different technology.) 	
We introduce $g(t)=H(t x'+(1-t)x,y)$, and thus we can write
$$
|F(x,y)-F(x',y)|^2 =| x-x'|^2
-2\lambda\left\langle
x-x',g(0)-g(1)
\right\rangle
+
\lambda^2 |g(0)-g(1)|^2,
$$
where
$$
g(0)-g(1) = \int_0^1 \partial_1 H (t x' + (1-t)x,y)(x-x')\,\mathrm{d}t.
$$
Using \eqref{eq:sconv}, we estimate $| g(0)-g(1)|\le M_y |x-x'|$
and
$\left\langle x-x',g(0)-g(1)\right\rangle\ge m_y |x-x'|^2$,
and arrive at
$$
|F(x,y)-F(x',y)|
\le
\left(
1+\lambda^2 M_y^2 -2\lambda m_y
\right)^{1/2}
|x-x'|.
$$
In \eqref{eq:sconv}, without the loss of generality, we can assume that $\mathbb{E}(m_{Z_1})<\infty$.
Furthermore, requiring $\mathbb{E}(M_y^2)<\infty$, we can set $\lambda$ such that conditions of Proposition \ref{itthm},
i.e \eqref{con}, are satisfied.
It is also not restrictive to assume that $\mathbb{E}(m_{Z_1})^2<\mathbb{E}(M_{Z_1}^2)$.
Since $K_{Z_1} = \left(1+\lambda^2 M_{Z_1}^2 -2\lambda m_{Z_1}\right)^{1/2}$, Jensen's inequality implies
$$\mathbb{E}(K_{Z_1})\le \left(1+\lambda^2 \mathbb{E} (M_{Z_1}^2) -2\lambda \mathbb{E}(m_{Z_1})\right)^{1/2}
< 1
$$
whenever $\lambda<\frac{2\mathbb{E}(m_{Z_1})}{\mathbb{E} (M_{Z_1}^2)}$.
\end{remark}

\begin{remark}
Strong convexity is a usual assumption in the stochastic gradient Langevin dynamics
literature, see \cite{dalalyan,eric,6}. Remark \ref{abo} above shows that our results
cover this case. It would, however, be nice to weaken this condition to dissipativity
(see \cite{5,lovas}). It seems that such a generalization requires much more
advanced techniques, see e.g.\ \cite{varvenne}.
\end{remark}

\begin{remark}
	Note that the iteration (\ref{rec}) in Section \ref{sec3} is a special case of the Langevin iteration \eqref{LLrec} such that
	\begin{equation}\label{buxt}
	H(X_{n},A_{n+1})=A_{n+1}X_{n}
	\end{equation}
	with the difference that for \eqref{buxt} only convexity holds and not strong convexity.
	In a least-squares regression setup it is important that $A_n$ is assumed positive semi-definite only
	and not necessarily positive definite.
\end{remark}

\section{Generalized multi-type Galton--Watson process}\label{sec6}

We follow the notation of Kevei and Wiandt \cite{kevei}.
{A $d$-type Galton--Watson branching process with immigration (GWI process)}
\[
\bX_n = ( X_{n,1} , \ldots, X_{n,d} ), \quad n \in \ZZ_+ ,
\]
is defined as
\[
\begin{cases}
\bX_n = \sum_{j=1}^{X_{n-1,1}} \bA_{n,j,1}+ \ldots + \sum_{j=1}^{X_{n-1,d}} \bA_{n,j,d}
+ \bB_n , \qquad n \geq 1 , \\
\bX_0 = v ,
\end{cases}
\]
where $v\in\mathbb{N}^{d}$,
$\{ \bA_{n,j,i}, \, \bB_n : n, j \in \NN , \, i \in \{1,\dots,d\} \}$
are  random vectors with non-negative integer coordinates.
Here $X_{n,i}$ is the number of $i$-type  individuals in the $n^\mathrm{th}$ generation
of a population, $\bA_{n,j,i}$ is the vector of the number of offsprings produced by the $j^\mathrm{th}$
individual of type $i$ belonging to the
$(n-1)^\mathrm{th}$ generation, and $\bB_n$ is the vector of the number of
immigrants.

Let $\bC_n:=\{ \bA_{n,j,i} :  j \in \NN , \, i \in \{1,\dots,d\} \}$.
In the standard setup  the families of random variables
$\{ \bC_n : n \in \NN\}$ and $\{ \bB_n : n \in \NN\}$ are independent and
$(\bC_n, \, \bB_n)$, $n \in \NN$ is a sequence of independent vectors.
The process $\{ \bX_n : n \in \NN\}$ is called homogeneous, when $(\bC_n, \, \bB_n)$, $n \in \NN$ are identically distributed, otherwise it is inhomogeneous.
In this section we study the generalization of the homogeneous case, when $\{ Z_{n}:=(\bC_n, \, \bB_n) : n \in \NN \}$ is
a stationary and ergodic process. As before, we extend this stationary process to the timeline $\mathbb{Z}$.
We furthermore assume for each $i$ that $\bA_{1,j,i}$ has, for each $j\in\mathbb{N}$, the same conditional law with
respect to $\mathcal{F}_{0}$.

Note that the state space of $Z_{n}$ is $\mathbb{R}^{\mathbb{N}}$. It can easily be checked that all the
arguments of our paper apply for such state spaces, too.

Homogeneous multi-type GWI processes has been introduced and studied by
Quine \cite{Quine, Quine2}. In \cite{Quine} necessary and sufficient condition is
given for the existence of stationary distribution in the subcritical case.
A complete answer is given by Kaplan \cite{kaplan}. Also Mode \cite{mode} gives
a sufficient condition for the existence of a stationary distribution, and
in a special case he shows that the limiting distribution is
a multivariate Poisson process with independent components.

Branching process models are extensively used in various parts of natural sciences,
in biology, epidemiology, physics, computer science, among other subjects.
In particular, multi-type GWI processes were used
to determine the asymptotic mean and covariance matrix of deleterious
and mutant genes in a stationary population by  Gladstien and Lange \cite{GlLa78},
and in non-stationary population by Lange and Fan \cite{LaFa97}.
Another rapidly developing area where multi-type GWI processes can be applied
is the theory of polling systems.
Resing \cite{Res} pointed out that a large variety of polling models can be
described as a multi-type GWI process.
Resing \cite{Res}, van der Mei \cite{Mei}, Boon \cite{boon1}, Boon et al. \cite{boon2}
and Altman, Fiems \cite{alt} investigated several communication protocols applied in
info-communication networks with differentiated services.
There are different quality of services,
for example, some of them are delay sensitive
(telephone, on-line video, etc.), while others tolerate some delay (e-mail, internet,
downloading files, etc.). Thus, the services are grouped into service classes such
that each class has an own transmission protocol like priority queueing. In the papers
mentioned above the $d$-type Galton--Watson process has been used, where the process was
defined either by the sizes of the active user populations of the $d$ service classes,
or by the length of the $d$ priority queues.
For the general theory and applications of multi-type Galton--Watson processes
we refer to Mode \cite{mode} and Haccou et al.~\cite{HJV}.

Define the random row vectors $m_i:=\EXP[\mathbf{A}_{1,1;i}\vert\mathcal{F}_0]$, $i=1,\ldots,d$,{}
where $\mathcal{F}_{t}$ is the sigma-algebra generated by $Z_{j}$, $-\infty<j\leq t$.
Note that, by our assumptions, $m_{i}$ has the same law as
$\EXP[\mathbf{A}_{n+1,j;i}\vert\mathcal{F}_n]$ for each $n$ and $j$.
For $x\in\mathbb{R}^d$ we will use the $\ell_1$-norm $|x|_1:=\sum_{j=1}^d |x_i|$
where $x_i$ is the $i$th coordinate of $x$.

\begin{prop}
If
\begin{equation}
\label{mos}
\max_{1\leq i\leq d}|m_i|_1\leq\varrho,
\end{equation}
almost surely for some constant $\varrho<1$; $\EXP\{|\mathbf{B}_{1}|_{1}\}<\infty$,
then
$\mathbf{X}_{t}=\mathbf{X}_{t}(v)$ is strongly stable and SLLN holds.	
\end{prop}
\begin{proof} Define
$$
F(x,z):=\sum_{j=1}^d\sum_{i=1}^{x_j}z_{i;j} +z_0,
$$
where
$$
z_{i;j}\in\mathbb{N}^{d},\ i\in\mathbb{N},1\leq j\leq d,\quad z_{0}\in\mathbb{N}^{d}.
$$
Note that this iteration is monotone.
As already defined, the stationary process will be
$$
Z_n=((\mathbf{A}_{n,i;j})_{1\leq j\leq d,i\in\mathbb{N}},\mathbf{B}_n).
$$

Let us check \eqref{Fo} with $p=1$:
\begin{eqnarray*}
& & \EXP[|F(x,Z_1)|_1\vert\mathcal{F}_0]\leq\sum_{j=1}^d \sum_{i=1}^{x_j}\EXP[|\mathbf{A}_{1,i;j}|_1\vert\mathcal{F}_0] +
\EXP[|\mathbf{B}_1|_1]\\
&\leq&  \sum_{j=1}^d x_j \varrho +\EXP[|\mathbf{B}_1|_1]= \varrho |x|_1+\EXP[|\mathbf{B}_1|_1].
\end{eqnarray*}
Note also that $\EXP[|F(0,Z_{1})|]=\EXP[|\mathbf{B}_{1}|_{1}]<\infty$, as required by \eqref{kezdoveges}.
We may conclude from Proposition \ref{Citthm}.
\end{proof}


\textcolor{black}{\begin{remark} In the case where the sequence $\mathbf{A}_{n,\cdot;\cdot}$, $n\in\mathbb{N}$ is i.i.d.\ one has
$m_i=\EXP[\mathbf{A}_{1,1;i}\vert\mathcal{F}_0]=\EXP[\mathbf{A}_{1,1;i}]$. In that case, the standard
assumption is that the matrix $M$ composed from row vectors $m_i$ satisfies
\begin{equation}\label{sept}
\varrho(M)<1,
\end{equation}
where $\varrho(M)$ denotes sepctral radius,
see \cite{kevei}. Our \eqref{mos} is stronger than \eqref{sept}. In arguments for i.i.d.
$\mathbf{A}_{n,\cdot;\cdot}$ the general
case \eqref{sept} can be easily reduced to \eqref{mos}. However, it is not clear to us how to do it in the current,
non-independent setting. Perhaps the techniques of Theorem 4 in \cite{DeTr21} could be adapted.
\end{remark}}

\section*{Acknowledgement}

We would like to express our sincere gratitude to the anonymous reviewer for the thorough and insightful evaluation of our paper. The careful examination not only identified errors that needed correction but also provided valuable suggestions that significantly enhanced the overall quality of the paper. We would also like to extend our thanks to the editorial team of the Journal of Applied Probability for facilitating this constructive review process. The thoughtful feedback received has undoubtedly strengthened the clarity and rigor of our work.

\end{document}